\documentclass[12pt,oneside]{amsart}

\usepackage{amssymb, amsthm, url, amscd}
\usepackage[capitalize]{cleveref} 
\usepackage{enumerate, amsfonts, latexsym,epsfig, color}
\usepackage{epstopdf,calc}

\usepackage{palatino}

\input xy
\xyoption{all}

\usepackage{amsmath}

\newcommand{\leftQ}[2]{\left.\raisebox{-.2em}{$#2$}\middle\backslash\raisebox{.2em}{$#1$}\right.}

\usepackage{color}

\newtheorem {theorem}{Theorem} [section]
\newtheorem {lemma} [theorem] {Lemma}
\newtheorem {proposition} [theorem] {Proposition}

\newtheorem {corollary} [theorem] {Corollary}
\newtheorem {definition} [theorem] {Definition}

\def\mc {\mathcal}

\def\Z {\mathbb Z}
\def\N {\mathbb N}
\def\co{\colon\thinspace}

\def\C {\mathbb C}

\def\H {\mathbb H}
\def\bS {\mathbb S}

\def\SO {\mathrm{SO}}

\def\immerse {\looparrowright}

\newcommand{\CAT} {\ensuremath{\operatorname{CAT}}}

\DeclareMathOperator{\arccosh}{arccosh}

\begin{document}

\title{Prescribed virtual homological torsion of 3-manifolds}

\author{Michelle Chu}
\author{Daniel Groves}

\address{Department of Mathematics, Statistics, and Computer Science,
University of Illinois at Chicago,
322 Science and Engineering Offices (M/C 249),
851 S. Morgan St.,
Chicago, IL 60607-7045}
\email{michu@uic.edu}
\email{groves@math.uic.edu}

\thanks{The first author is supported in part by NSF grant DMS 1803094 and DMS 1928930 while the author participated in a MSRI program during the Fall 2020 semester.  The second author is supported in part by NSF grant DMS 1904913.}

\begin{abstract}
We prove that given any finite abelian group $A$ and any irreducible $3$--manifold $M$ with empty or toroidal boundary which is not a graph manifold there exists a finite cover $M' \to M$ so that $A$ is a direct factor in $H_1(M',\Z)$.  This generalizes results of Sun \cite{Sun-VHT} and of Friedl--Herrmann \cite{FH}. 
\end{abstract}

\maketitle

\section{Introduction}
In \cite{Sun-VHT}, Sun showed that any closed hyperbolic $3$--manifold virtually contains any prescribed finite subgroup in homological torsion.
Sun used the immersed almost-Fuchsian surfaces of Kahn and Markovic \cite{KM12} to construct immersed $\pi_1$--injective $2$--complexes. By using Agol's result that the fundamental groups of closed hyperbolic $3$--manifolds are virtually compact special \cite{Agol} and the implications on virtual retractions to quasi-convex subgroups, for any closed hyperbolic $3$--manifold Sun \cite[Theorem 1.5]{Sun-VHT} finds a finite cover containing the prescribed finite abelian group as a direct factor in homology. 

Since the Kahn-Markovic construction requires that the manifolds be closed, Sun's results do not apply to hyperbolic $3$--manifolds with cusps. Indeed, Sun asked whether his result applied also to finite-volume hyperbolic 3-manifolds with cusps. In this paper, we extend the results of Sun to a larger class of $3$--manifolds which includes all finite-volume hyperbolic $3$--manifolds, giving a positive answer to \cite[Question 1.8]{Sun-VHT}.

\begin{theorem}\label{t:main}
Suppose that $M$ is an irreducible $3$--manifold with empty or toroidal boundary which is not a graph manifold and that $A$ is a finite abelian group. There is a finite cover $M' \to M$ so that $H_1(M';\mathbb{Z})$ has a direct factor isomorphic to $A$.
\end{theorem}

Prior to \cref{t:main}, Friedl and Herrmann used \cite{Sun-VHT} and a result of Hadari \cite{Hadari} to show that for any such $M$ and any $k>0$ there is finite cover $N \to M$ with $\left|H_1(N;\mathbb{Z})\right|>k$ \cite[Theorem 1.3]{FH}. Independently, Liu showed that any such $M$ admits a finite \emph{regular} cover $N' \to M$ with $\left|H_1(N';\mathbb{Z})\right|\neq0$ \cite[Corollary 1.4]{Liu}. 

A \emph{hybrid} hyperbolic manifold is constructed either by inbreeding (c.f. \cite{Agolhyb,BTsys}) or interbreeding (c.f. \cite{GPS}) arithmetic hyperbolic manifolds.
For $n>3$ every arithmetic hyperbolic $n$--manifold $N$ of simplest type contains a totally geodesic arithmetic hyperbolic $3$--manifold $M$ (coming from restrictions of the associated quadratic form). By \cite[\S9]{BHW}, we get the following corollary (some of these cases follow from \cite{Sun-VHT}).

\begin{corollary}
Suppose that $n>3$ and $N$ is a finite-volume hyperbolic $n$--manifold which is either arithmetic of simplest type or a hybrid. If $A$ is a finite abelian group then there is a finite cover $N'\to N$ so that $H^1(N';\mathbb{Z})$ has a direct factor isomorphic to $A$.
\end{corollary}

The bulk of this paper is devoted to the case of \cref{t:main} where $M$ is a finite-volume hyperbolic $3$--manifold. We follow the strategy of \cite{Sun-VHT} but give an independent proof which simplifies and generalizes Sun's arguments, recovering Sun's results in the closed hyperbolic setting. We replace Sun's use of the results of Kahn and Markovic \cite{KM12} with those of Kahn and Wright \cite{KW} and replace some arguments of Sun with an elementary argument using coverings of surfaces. 
We begin in \cref{s:QI embeddings} by recording some facts about quasi-isometries and hyperbolic spaces. In \cref{s:KW surfaces} we apply the construction of Kahn-Wright to build an almost-Fuchsian surface in $M$. In \cref{ss:Xn} we use the Kahn-Wright surface to construct a 2-complex $X_n\immerse M$. 
We then apply virtual retraction properties to complete the proof of \cref{t:main} in the finite-volume hyperbolic case in \cref{s:main-hyp}. Finally, in \cref{s:mixed} we deduce the general case of \cref{t:main} from that of finite-volume hyperbolic  $3$--manifolds. 

We remark that independent from Kahn and Wright, Cooper and Futer \cite{CF} obtained similar results on constructing many closed immersed $\pi_1$--injective quasi-Fuchsian surfaces in finite-volume hyperbolic $3$--manifolds with cusps. However, our arguments rely on the additional control on the quasi-conformal constants and on the holonomies in the Kahn-Wright constructions.

\section{Quasi-isometric embeddings}\label{s:QI embeddings}
In this section we record some elementary facts about quasi-isometries and hyperbolic spaces.

\begin{definition}
Let $k, \lambda, c$ be constants, and let $X,Y$ be metric spaces.  A map $f \co X \to Y$ is a {\em $k$--local $(\lambda,c)$--quasi-isometric embedding} if for all $x \in X$ the restriction to the ball of radius $k$ 
\[	f |_{B_k(x)} \co B_k(x) \to Y	\]
is a $(\lambda,c)$--quasi-isometric embedding.
\end{definition}

The following is essentially \cite[Theorem A.20]{KW}.
\begin{proposition} \label{prop:local-to-global}
For all $\delta$, for all $c \ge 0$ and all $\lambda \ge 1$ there exist $k,\lambda',c'$ so that if $Y$ is a $\delta$--hyperbolic metric space and $X$ is a geodesic metric space then any $k$--local $(\lambda,c)$--quasi-isometric embedding is a $(\lambda',c')$--quasi-isometric embedding.
\end{proposition}
\begin{proof}
Since $X$ and $Y$ are geodesic metric spaces, distances in $X$ and $Y$ are calculated by geodesics.  Therefore, we can apply the standard local-to-global result for quasi-geodesics (see, for example, \cite[Theorem 3.1.4, p.25]{cdp}).
\end{proof}

\subsection{Half-planes}

Let $\theta \in (0,\pi]$.  Let $P_\theta$ be the subspace of $\H^3$ obtained from gluing two totally geodesic half-planes together along their boundary geodesic, meeting at angle $\theta$.  There is a natural embedding $p_\theta \co \H^2 \to \H^3$ taking $\H^2$ to $P_\theta$ given by mapping the imaginary axis to the boundary geodesic of the two half-planes (we consider $\H^2$ in the upper half-space model as a subset of $\C^2$).  The image of these boundary geodesics is the \emph{pleating locus} for $p_\theta$.

\begin{lemma} \label{lem:glue half-planes}
Given $\theta \in (0,\pi]$ there exists $c_\theta \ge 0$ so that for all $\eta \in [\theta,\pi]$ the map $p_{\eta}$ is a $(1,c_\theta)$--quasi-isometric embedding.
\end{lemma}
\begin{proof}
We show that it suffices to take 
$$c_\theta = 2 \cdot \arccosh \left( \frac{1}{\sin\left(\frac{\theta}{2}\right)} \right)	.	$$
Indeed, suppose that $x,y \in \H^2$ and let $\overline{x} = p_\eta(x)$ and $\overline{y} = p_\eta(y)$, and consider the image of $[x,y]$ in $p_\eta(\H^2)$.  If the sign of the real parts of $x$ and $y$ are the same, then $[x,y]$ maps to a geodesic in $\H^3$ and $d_{\H^3}(\overline{x},\overline{y}) = d_{\H^2}(x,y)$ in this case.

Suppose then that the signs of the real parts of $x$ and $y$ are different, and let $z\in \H^2$ be the point where $[x,y]$ meets the $y$--axis.  Let $\overline{z} = p_\eta(z)$.  Then $p_\eta([x,y])$ consists of two geodesic segments $[\overline{x},\overline{z}]$ and $[\overline{z},\overline{y}]$ meeting at some angle $\alpha \ge \eta \ge \theta$.

Consider the geodesic triangle $\Delta$ in $\H^3$ with vertices $\overline{x},\overline{y},\overline{z}$, and let $e$ be the distance from $\overline{z}$ to the geodesic $\gamma = [\overline{x},\overline{y}]$.  The geodesic from $\overline{z}$ to $\gamma$ cuts $\Delta$ into two right-angled hyperbolic triangles, one of which has angle at $\overline{z}$ at least $\frac{\theta}{2}$.  We thus have a hyperbolic triangle with side lengths $e,a,b$, say, where the angle opposite $b$ is $\frac{\pi}{2}$, and the angle opposite $a$ is at $\overline{z}$ and is $A \ge \frac{\theta}{2}$.  Let $E$ be the angle opposite the side of length $e$.

\begin{figure}
  \begin{center}
  \fontsize{10pt}{10pt}\selectfont
  \def\svgwidth{0.50\textwidth}
   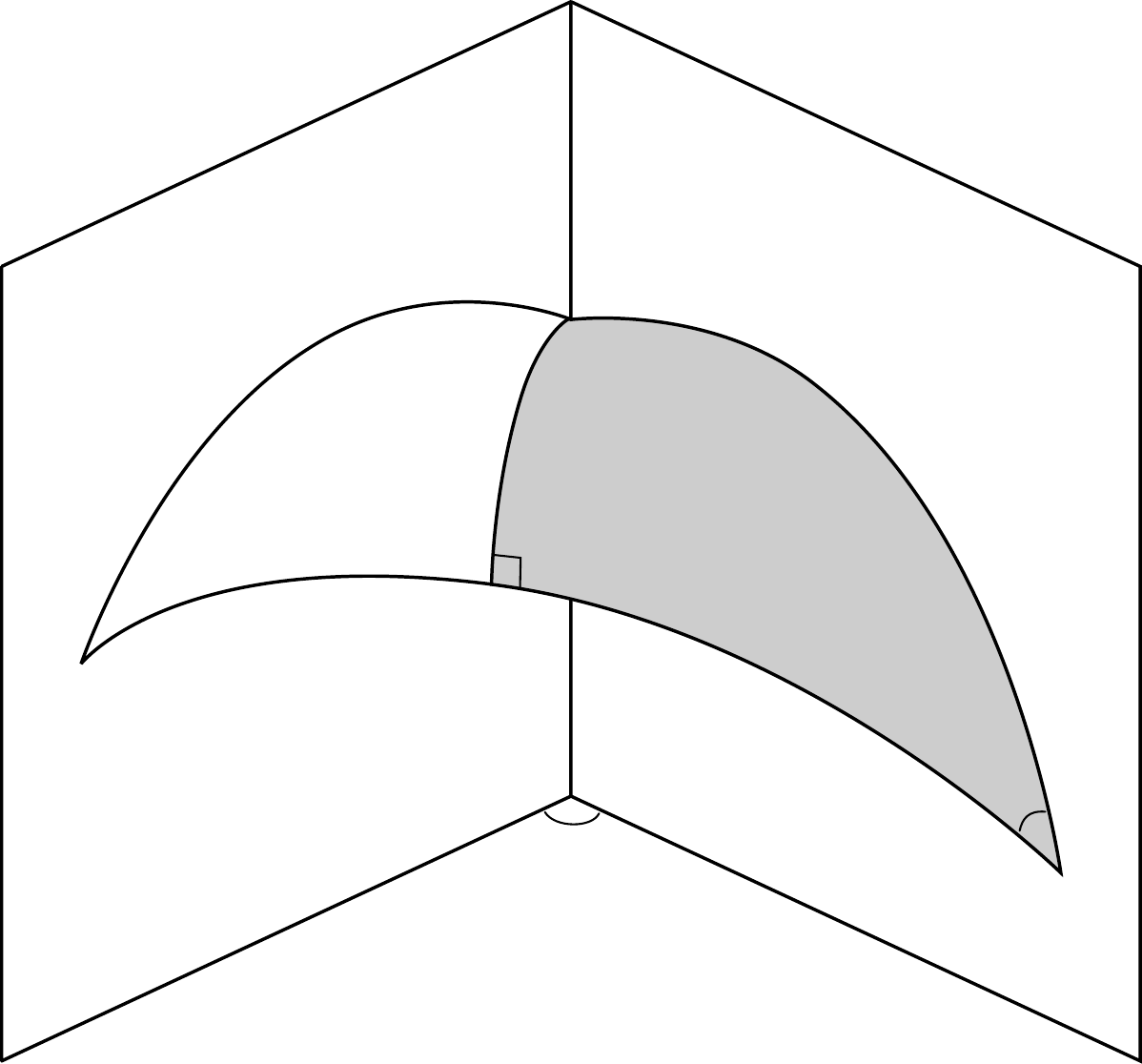
    \caption{The proof of \cref{lem:glue half-planes}}
\end{center}
\end{figure}

The second hyperbolic law of cosines says
\[	\cos(E) = -\cos(A)\cos\left(\frac{\pi}{2}\right) + \sin(A)\sin\left(\frac{\pi}{2}\right) \cosh(e)	,	\]
so
\[	\cosh(e) =  \frac{\cos(E)}{\sin(A)} \le \frac{1}{\sin\left( \frac{\theta}{2}\right)} .	\]
Let $d_1 = d_{\H^3}(\overline{x},\overline{z})$ and $d_2 = d_{\H^3}(\overline{z},\overline{y})$.  Observe that $d_{\H^2}(x,y) = d_1 + d_2$.  It is clear that
\[	d_1+d_2 -2 \arccosh \left( \frac{1}{\sin\left( \frac{\theta}{2}\right)}\right)  \le d_{\H^3}(\overline{x},\overline{y}) \le d_1 + d_2	,	\]
and the result follows.
\end{proof}

\section{Kahn-Wright surfaces}\label{s:KW surfaces}
From this section until the end of \cref{s:main-hyp}, let $N$ be a finite-volume hyperbolic $3$--manifold.  We remark that the arguments work in the closed setting as well as in the cusped setting, hence recovering Sun's results.

The set of closed geodesics in $N=\H^3/\Gamma$ is in 1-to-1 correspondence with the set of conjugacy classes of loxodromic elements in $\Gamma$. For a closed geodesic $\alpha$ in $N$ (with corresponding conjugacy class $[\gamma]\subset\Gamma$) let $\ell(\alpha)$ denote the length of $\alpha$ (the translation length $\gamma$) and $\theta(\alpha)$ the holonomy class of $\alpha$ (the rotation angle of $\gamma$ around its axis). 

\subsection{Pre-good curves}
Later in the section, we give a brief discussion of the construction of surfaces due to Kahn and Wright in \cite{KW}.  However, we first give a lemma which proves the existence of certain well-behaved geodesics whose $n^{\mathrm{th}}$ powers will become part of the Kahn--Wright surface.  
See \cite[$\S3$]{KW} for the definition of height in the following statement.  The following is an analogue in the finite-volume case of Sun's \cite[Lemma 2.9]{Sun-VHT}.  In order to use this geodesic in Kahn and Wright's construction, it is important to control the height.

\begin{lemma}\label{lem:geodesic}
For $n\in\N, \epsilon>0$, $h>0$, there exists $R_0$ so that for all $R > R_0$ there exists a geodesic $\alpha_0$ in $N$ of height at most $h$ such that
$\left| \ell(\alpha_0)-\frac{2R}{n} \right|<\frac{\epsilon}{n}$ and $\left| \theta(\alpha_0)-\frac{2\pi}{n} \right|<\frac{\epsilon}{n}$.
\end{lemma}

\begin{proof}
For a closed subset $\Omega$ of $\SO(2)$ and $T > 0$, let
$$\mc{G}(T,\Omega)=\left\{\alpha : \alpha \text{ is a closed geodesic in } N, \ell(\alpha)\leq T, \theta(\alpha)\in\Omega \right\} . $$

As noted in \cite[\S 3.1]{KW}, an application of the Margulis argument shows that
\begin{equation}
\#\mc{G}(T,\Omega) \sim \frac{e^{2T}}{2T} \|\Omega\| \text{ as }T\rightarrow\infty
\end{equation}
which in this case follows, for example, from \cite[Theorem 1.1]{MMO14} by setting $\varphi:=1_{\Omega}$ the indicator function on $\SO(2)$ (see also \cite{GangolliWarner}).

Considering geodesics $\alpha\in \mc{G}(2R/n+\epsilon/n,\Omega)\setminus \mc{G}(2R/n-\epsilon/n,\Omega)$ where $\Omega$ is the interval $(\frac{2\pi}{n}-\frac{\epsilon}{n},\frac{2\pi}{n}+\frac{\epsilon}{n})$,
we have 
\begin{equation}
\#\left\{ \alpha : \left| \ell(\alpha_0)-\frac{2R}{n} \right|<\frac{\epsilon}{n} \text{ and } \left| \theta(\alpha_0)-\frac{2\pi}{n} \right|<\frac{\epsilon}{n} \right\} \sim c_\epsilon\frac{e^{4R}}{4R} .
\end{equation}

The arguments in the proof of \cite[Lemma 3.1]{KW} apply to show that as $R$ grows, the proportion of those $\alpha$ with height larger than $h$ shrinks. In particular, for sufficiently large $R$ one can find $\alpha_0$ as needed. 
\end{proof}

Note that $\alpha_0$ may be chosen to be primitive.   In the language of Kahn and Wright, $\alpha_0^n$ is an \emph{$(R,\epsilon)$--good curve}.

\begin{definition}
Fix $n \in \N$, and also $R, \epsilon$.  An $(R,\epsilon,n)$--pre-good curve in $N$ is a geodesic $\alpha_0$ satisfying the conclusion of \cref{lem:geodesic} for some $h \in (0,1)$.
\end{definition}
We remark that Kahn and Wright allow curves to have height at most $50\log(R)$ before needing to be ``cut-off''.  We assume that $R > e$ so certainly curves of height less than $1$ are fine.  \cref{lem:geodesic} asserts that for fixed $n$ and $\epsilon$, for large enough $R$ there exists an $(R,\epsilon,n)$--pre-good curve (in fact there are many).

\subsection{The construction of Kahn and Wright}

In \cite{KW}, Kahn and Wright build certain quasi-Fuchsian immersed surfaces in $N$ out of pieces called \emph{good pants} and \emph{umbrellas}.  In turn, the umbrellas are assembled out of \emph{good hamster wheels}.  Each good pant and good hamster wheel is immersed in $N$, and has geodesic boundary components, which are referred to as \emph{cuffs}.

The construction in \cite{KW} depends on choices of parameters $R$ (sufficiently large) and $\epsilon > 0$ (sufficiently small).  We postpone for now the choice of these parameters to discuss the construction.  Kahn and Wright also specify another pair of parameters called \emph{cutoff heights}, and the purpose of \cref{lem:geodesic} above is to ensure that we can find an $\alpha_0$ whose height stays below the cutoff heights, and whose $n^{\mathrm{th}}$ power is a good curve.

Suppose that we find a curve $\alpha_0$ as in \cref{lem:geodesic} so that $\alpha = \alpha_0^n$ is an $(R,\epsilon)$--good curve, and so that the height of $\alpha_0$ (and hence $\alpha$) is at most $1$.  Then Kahn and Wright build a surface $S$ out of good pants and good hamster wheels, and $\alpha$ appears as a cuff on at least two (in fact many) of these pieces.

Consider $\alpha_0 \co \bS^1 \to N$ as a map from the circle to $N$ parametrized proportional to arc length, and let $u = \alpha_0(1)$.  Let $\phi_n \co \bS^1 \to \bS^1$ be the connected $n$--fold covering map, and let $\alpha \co \bS^1 \to N$ be the composition $\alpha_0 \circ \phi_n$.
Suppose $f \co S_0 \immerse N$ is a Kahn--Wright surface and that there is a map $c \co \bS^1 \to S_0$ so $\alpha = f \circ c$.  Let $\phi_n^{-1}(u) = \{ u_1, \ldots , u_n \}$ and note that there are $n$ different points $\{ x_1, \ldots , x_n \}$ on $S_0$ so $f(x_i) = \alpha(u_i)$ for each $i$.

Choose a basepoint $y \in \H^3$ and let $\pi \co  (\H^3,y) \to (N,u)$ be the based universal covering map.  Fix a basepoint $z \in \H^2$, and for each $i$, let $\tau_i \co (\H^2, z) \to (S_0,x_i)$ be a based universal cover.  

The map $f$  elevates to $n$ distinct (based) maps:
\[	\widetilde{f_i} \co (\H^2,z) \to (\H^3,y)	\]	
so that for each $i$ we have $\widetilde{f_i} \circ \pi = \tau_i \circ f$.

Let $H_R  = \{ z \in \H^2 \mid \rm{Re}(z) \ge 0\}$ and $H_L = \{ z \in \H^2 \mid \rm{Re}(z) \le 0 \}$, and let $m= \{ z \in \H^2 \mid \rm{Re}(z)  = 0 \} = H_R \cap H_L$.

Now, for a pair $i \ne j$ from $\{ 1 ,\ldots , n \}$ we have $\widetilde{f_i}(m) = \widetilde{f_j}(m)$.  Thus, we can take the two maps $\widetilde{f_i}|_{H_R}$ and $\widetilde{f_j}|_{H_R}$ and glue them together via an orientation-preserving isometry along the boundary to get a continuous map $\widetilde{f}_{i,f}^{H_R} \co \H^2 \to \H^3$, and similarly for the two maps restricted to $H_L$ to get a continuous map $\widetilde{f}_{i,j}^{H_L}$.

Kahn and Wright prove that for appropriate choices of parameters, their surface, built as a \emph{good assembly} of pants and hamster wheels, is close to a {\em perfect assembly}, and that the map which takes the good assembly to the perfect assembly is \emph{compliant} (see \cite[$\S$A.5]{KW}), which in particular means that it takes cuffs to cuffs.  For a perfect assembly with cuff $\alpha$, the construction analogous to the $\widetilde{f}_{i,j}$ leads to pairs of totally geodesic half-planes glued along their boundary geodesic, namely to a map $p_\theta$ for some $\theta$.  Thus, the map that takes the good assembly to the perfect assembly induces a map between $\widetilde{f}_{i,j}^{H_R} \co \H^2 \to \H^3$ and some map $p_\theta \co \H^2 \to \H^3$, and this map takes $m$ to the pleating locus for $p_\theta$.  

Our first task is to bound $\theta$ away from $0$, and our second is to show that the two maps are close.  The sense in which they are close will be that of \cite[p.51]{KW} -- being of \emph{$\epsilon_0$--bounded distortion to distance $D$} for appropriate choice of $\epsilon_0$ and $D$.

Denote the angles of the maps $p_\theta$ induced by $i$, $j$ and $H_R$ by $\theta(i,j,H_R)$, and for $i$, $j$ and $H_L$ by $\theta(i,j,H_L)$.

The following is a summary of the above discussion, and also of \cite[Theorem A.18]{KW}. Note that it follows from the proof of \cite[Theorem A.18]{KW} that the chosen maps which we denote by $g_{i,j}^{H_R}$ and $g_{i,j}^{H_K}$ in the following statement are compliant. In the following statement, $R_0$ is the constant from the statement of \cref{lem:geodesic}.

\begin{theorem} \label{t:KW summary}
Fix $n \in \N$.
For all $D$ there exist $C$, $\epsilon_0$ and $R_1 > R_0$ so that for all $\epsilon \in (0,\epsilon_0)$ and all $R > R_1$ and any $(n,R,\epsilon)$--pre-good curve $\alpha_0$ there exists a Kahn-Wright surface $f \co S_0 \immerse N$ containing $\alpha = \alpha_0^n$ as a cuff, constructed as an $(R,\epsilon)$--good assembly.

For each $i,j \in \{ 1,\ldots, n \}$ with $i \ne j$, there are maps $g_{i,j}^{H_R}, g_{i,j}^{H_L} \co \H^3 \to \H^3$ so that 
$\widetilde{f}_{i,j}^{H_R} \circ g_{i,j}^{H_R} = p_{\theta(i,j,H_R)}$ and $\widetilde{f}_{i,j}^{H_L} \circ g_{i,j}^{H_L} = p_{\theta(i,j,H_L)}$.
Moreover, we have:
\[	\theta(i,j,H_R), \theta(i,j,H_L) \in \left(\frac{\pi}{n},\pi\right)		,	\]
and $g_{i,j}^{H_R}$ is a $D$--local $(1+C\epsilon,C\epsilon)$--quasi-isometry.
\end{theorem}

The following is an easy consequence of \cref{t:KW summary} and \cref{lem:glue half-planes}.  In the following statement $R_0$ is the constant from \cref{lem:geodesic} and $R_1$ is the constant from \cref{t:KW summary}.
\begin{corollary} \label{c:local QI}
Fix $n \in \N$.  There exist $\lambda,\kappa$ so that for any $D$ there exist  $R_D > R_1, R_0$ and $\epsilon_D > 0$, so that for any $\alpha_0$ and $f \co S_0 \immerse N$ as in \cref{t:KW summary} with $R > R_D$ and any $\epsilon \in (0,\epsilon_D)$ the maps $\widetilde{f}_{i,j}^{H_R}$ and $\widetilde{f}_{i,j}^{H_L}$ are $D$--local $(\lambda,\kappa)$--quasi-isometric embedding.
\end{corollary}

Now, choose $D, \lambda_1,\kappa_1$ so that any $D$--local $(\lambda,\kappa)$--quasi-isometric embedding from $\H^2$ to $\H^3$ is a global $(\lambda_1,\kappa_1)$--quasi-isometric embedding (see \cref{prop:local-to-global}).  This $D$ then gives $R_D$ and $\epsilon_D$ as above.  

\cref{lem:geodesic} proves that there is an $(n,R_D,\epsilon_D)$--pre-good curve $\alpha_0$, and the construction from \cite{KW} proves that there is an $f \co S_0 \to N$ with $\alpha = \alpha_0^n$ as a cuff satisfying the conclusions of \cref{t:KW summary} and \cref{c:local QI}, with $R= R_D$.

We fix this map $f \co S_0 \immerse N$, along with $n$, $D$, $R_D$, $\epsilon_D$, $\alpha_0$, $\alpha=\alpha_0^n$, $\kappa$, and $\epsilon$ as chosen above for the next two sections.

\section{The space $X_n$} \label{ss:Xn}
By standard separability properties of surface groups, we may find a cover $S \to S_0$ to which $\alpha$ lifts as a non-separating simple closed curve, and so that:

\begin{enumerate}
\item The injectivity radius of $S$ is at least $\max\{ 2D, \lambda_1\kappa_1 \}$; and
\item The lift of $\alpha$ to $S$ is contained in an embedded collar of width at least $\max \{ 2D, \lambda_1\kappa_1 \}$.
\end{enumerate}

Given the surface $S$, we build a space $X_n$ which immerses into $N$, exactly as in \cite{Sun-VHT}.  Passing from $S_0$ to $S$ before constructing $X_n$ makes the proof that $X_n$ is $\pi_1$--injective with quasi-convex image much simpler than Sun's proof from \cite[$\S4$]{Sun-VHT}.
Let $C$ denote the image of $\alpha$ in $S$, and let $\phi^n_C \co C \to \mathbb{S}^1$ be an $n$--to--$1$ covering map, and let $\tau_C \co C \to C$ be a deck transformation.  We may choose $\phi^n_C$ so that $\tau$ is an isometry.

\begin{definition}
The space $X_n(S,C)$ is defined by cutting $S$ along $C$ to get a surface $S_1$ with two boundary components, denoted $C_1$ and $C_2$, and taking the quotient of $S_1$ by the relation generated by $c \sim \tau_{C_i}(c)$ for $c \in C_i$ and $i = 1,2$.
\end{definition}

Suppose that $S$ is equipped with a hyperbolic metric, and consider the induced metric on $S_1$.  Since the maps $\tau_{C_i}$ are isometries, there is a natural induced quotient metric on $X_n(S,C)$, which is locally isometric to $\H^2$ away from the images of the $C_i$.

The following result is clear from the construction of $S$ from $S_0$.
\begin{lemma}
The injectivity radius of $X_n$ is at least $\max \{ 2D, \lambda_1\kappa_1 \}$.
\end{lemma}

Let $S_1$ be the surface obtained from $S$ by cutting along $C$, and let $C_1, C_2$ be the boundary components of $S_1$.  Let $q \co S_1 \to X_n$ be the defining quotient map and let $\overline{C_i} = q(C_i)$ for $i = 1,2$.  

Because in $S$ the curve $C$ has an embedded collar of width at least $2D$, for any $i,j \in \{ 1, 2\}$, any two distinct elevations of $\overline{C_i}$ and $\overline{C_j}$ to $\widetilde{X_n}$ are at distance at least $4D$ from each other.

\begin{definition}
Suppose that $\mc{A} = \{ Z_1, \ldots , Z_m \}$ is a finite collection of metric spaces and that $k > 0$.  A metric space $Z$ is {\em $k$--modeled on $\mc{A}$} if for every $z \in Z$ there is an $i$ so that the ball of radius $k$ about $z$ is isometric to a ball in $Z_i$.
\end{definition}

Recall $H_R = \{ z \mid \mathrm{Re}(z) \ge 0 \}$ is the (closed) half-hyperbolic plane (in the upper half-space model).  Let $W_n$ be the space obtained from $n$ copies of $H_R$ glued along the boundary geodesics (by an isometry).

\begin{lemma}
The space $\widetilde{X_n}$ is $D$--modeled on $W_n$.
\end{lemma}
\begin{proof}
Let $x \in \widetilde{X_n}$ and consider the covering map $\pi \co \widetilde{X_n} \to X_n$.

{\bf Case 1:} $d(\pi(x),\{ \overline{C_1}, \overline{C_2} \} ) \le D$.

In this case, in $\widetilde{X_n}$ there is a unique elevation of some $\overline{C_i}$ which lies within $D$ of $x$.  Let $y$ be a point in this elevation so $d(x,y) \le D$.  Then $B_{D}(x) \subseteq B_{2D}(y)$, and $B_{2D}(y)$ is isometric to a ball of radius $2D$ in $W_n$.

{\bf Case 2:}  $d(\pi(x), \{ \overline{C_1}, \overline{C_2} \} ) > D$.

In this case there is no elevation of either $\overline{C_i}$ which lies within $D$ of $x$, and $B_{D}(x)$ is isometric to a ball of radius $D$ in $\H^2$ (and so in $W_n$).
\end{proof}

By construction the immersion $f_1 \co S \to N$ obtained from composing the covering map $S \to S_0$ with $f \co S_0 \looparrowright N$ yields an immersion $g \co X_n \looparrowright N$.  Let $\widetilde{g} \co \widetilde{X_n} \to \H^3$ be the induced map on universal covers.

Two points $x, y$ in $\widetilde{X_n}$ at distance at most $D$ either lie in an isometrically embedded copy of a half-space from $\H^2$, or else in two different ``sheets" of a copy of $W_n$.  In either case, it follows immediately from \cref{c:local QI} that $\widetilde{g}$ is a $D$--local $(\lambda,\kappa)$--quasi-isometric embedding.  Thus, 

\begin{theorem} \label{t:Xn QC}
The map $\widetilde{g} \co \widetilde{X_n} \to \H^3$ is a $D$--local $(\lambda,\kappa)$--quasi-isometric embedding, and hence is a (global) $(\lambda_1,\kappa_1)$--quasi-isometric embedding.

In particular, since the injectivity radius of $X_n$ is at least $\lambda_1\kappa_1$ the map $g$ is $\pi_1$--injective, and $g_\ast(\pi_1(X_n))$ is relatively quasi-convex in $\pi_1(N)$.
Moreover, $g_\ast(\pi_1(X_n))$ does not intersect any (conjugate of) the cusp subgroups of $\pi_1(N)$.
\end{theorem}

\section{Virtual retractions and the proof of \cref{t:main} in the hyperbolic case} \label{s:main-hyp}

In this section we prove \cref{t:main} in case of a finite-volume hyperbolic $3$--manifold $N$.
Let $\Gamma = \pi_1(N)$.  By \cite[Theorem 1.1]{Agol} and \cite[Theorem 14.29]{Wise} $\pi_1(N_0)$ is the fundamental group of a virtually special cube complex $X$.  Let $\Gamma_1 \le \Gamma$ be a finite-index subgroup so that the cover of $X$ corresponding to $\Gamma_1$ is special, and let $N_1$ be the cover of $N$ corresponding to $\Gamma_1$.  As in \cref{ss:Xn}, construct an immersion $g \co X_n \to N_1$.  Note that $(\Gamma_1,\mc{P})$ is relatively hyperbolic where $\mc{P}$ consists of the (abelian) cusp subgroups.  

Let $H = g_\ast\left( \pi_1(X_n) \right) \le \pi_1(N_1) = \Gamma_1$.
The subgroup $H$ is relatively quasi-convex in $\Gamma_1$, and so by \cite[Corollary 6.7]{HW08} (with the formulation as in \cite[Theorem 6.3]{PW}) we have that $H$ is a virtual retract of $\Gamma_1$.
Let $\Gamma_2$ be a finite-index subgroup of $\Gamma$ which retracts onto $H$.
Let $N_2$ be the finite cover of $N_1$ corresponding to $\Gamma_2$. 
As in \cite[Proposition 3.7]{Sun-VHT}, we have the induced maps on homology
\[ H_1(X_n;\Z) \overset{g_*}\longrightarrow H_1(N_2;\Z) \overset{r_*}\longrightarrow H_1(X_n;\Z). \]
Therefore, since $r\circ g_*=id_H$, $H_1(X_n;\Z)=\Z / n\Z\oplus \Z^{k}$ is a direct factor of $H_1(N_2;\Z)$. In particular, $\Z / n\Z$ is a direct factor of $H_1(N_2;\Z)$ and this proves the hyperbolic case of  \cref{t:main} in the case that $A$ is finite cyclic.

Given a finite abelian group $A$, induction on the rank $k$ of $A$ also works as in \cite[Proposition 3.9]{Sun-VHT} as follows. Let $A=\oplus_{i=1}^{k-1}\Z/n_i\Z$ and $A'=A\oplus \Z/n_{k}\Z$.  Suppose by induction that $H \le \Gamma_1$ is a relatively quasi-convex free product of images of $\pi_1\left( X_{n_i} \right)$ (for $i = 1, \ldots , k-1$) and that $H'=(g_k)_*\left( \pi_1(X_{n_k})\right) \le \Gamma_1$.  Choose any $\gamma\in\Gamma_1$ whose fixed points in $\partial\H^3$ are disjoint from both limit sets $\Lambda(H)$ and $\Lambda(H')$. Then after conjugating $H'$ by some sufficiently high power $\gamma^m$, the first Klein--Maskit combination theorem \cite{MaskitIV} applies (note that by \cite[Corollary 1.3]{Hruska2010} a subgroup is relatively quasi-convex if and only if it is geometrically finite). Since $H'$ is relatively quasi-convex, the free product $H*gH'g^{-1}$ is also a relatively quasi-convex subgroup of $\Gamma_1$ isomorphic to the abstract group $H*H'$. 
The proof now follows exactly as above, completing the proof of \cref{t:main} in the finite-volume hyperbolic case.

\section{Non-hyperbolic manifolds} \label{s:mixed}

We now prove \cref{t:main} in general.  To that end, suppose that $M$ is an irreducible $3$--manifold which is not a graph manifold, and that $A$ is a finite abelian group.  By \cite[Theorem 1.1]{PW} there exists a \CAT$(0)$ cube complex $X$ equipped with a free $\pi_1(M)$--action so that there are finitely many orbits of hyperplanes and so that $\leftQ{X}{\pi_1(M)}$ has a finite special cover.  Let $\Gamma_1 \le \pi(M)$ be a finite-index subgroup corresponding to the finite special cover of $\leftQ{X}{\pi_1(M)}$, and let $M_1$ be the finite cover of $M$ corresponding to $\Gamma_1$.  
  Let $M_h$ be a hyperbolic piece in the geometric decomposition of $M_1$ and let $\Gamma_h := \pi_1(M_h) \le \pi_1(M_1)$ (basepoints/conjugacy classes are not important here).  According to the construction in the previous sections, there is a relatively quasi-convex subgroup $H$ of $\Gamma_h$ so that $A$ is a direct factor of $H^1(H;\Z)$.  \cref{t:main} follows immediately from the following result. This result is presumably known to the experts, but we were unable to find it in the literature.

\begin{proposition}
There is a finite-index subgroup $\Gamma_2 \le \Gamma_1$ so that $H \le \Gamma_h$ and $H$ is a retract of $\Gamma_2$.
\end{proposition}
\begin{proof}
Let $\widetilde{M_h} \le \widetilde{M}$ be the ($\Gamma_h$--invariant) universal cover of $M_h$ inside the universal cover of $M$.  The space $X$ is built via a \emph{wallspace} construction  on $\widetilde{M}$ as in \cite{hruskawise:finiteness}.  As in \cite[$\S$2.9]{hruskawise:finiteness}, we can associate to $\widetilde{M_h}$ a \emph{hemiwallspace} consisting of those half-spaces in $\widetilde{M}$ which intersect $\widetilde{M_h}$.  This builds a $\Gamma_h$--invariant convex sub-complex $X_h$ of $X$ by \cite[Lemma 2.29]{hruskawise:finiteness}.

By \cite[Theorem 2.1]{PW}, the surfaces of the cubulation intersecting $M_h$ all intersect $M_h$ in a geometrically finite surface.  Therefore, by \cite[Theorem 6.12]{hruskawise:finiteness} the $\Gamma_h$--action on $X_h$ is (free and) \emph{co-sparse} (see \cite[Definition 7.1]{SageevWise15}).

According to \cite[Theorem 7.2]{SageevWise15} inside of $X_h$ there is an $H$--invariant convex sub-complex $Z$ upon which $H$ acts co-sparsely.  In fact, since $H$ does not intersect any of the parabolic subgroups of $\Gamma_h$, the sub-complex $Z$ found in \cite[Theorem 7.2]{SageevWise15} is $H$--cocompact (this follow immediately from the proof).

Since $\leftQ{X}{\Gamma_1}$ is special, it follows from \cite[Corollary 6.7]{HW08} (we use the formulation as in \cite[Theorem 6.3]{PW}) that $H$ is a virtual retract of $\Gamma_1$.
\end{proof}

\small
\bibliography{VirtualTorsionarxiv}
\bibliographystyle{alpha}

\end{document}